\documentclass[preprint]{elsarticle}

\usepackage{hyperref}

\journal{JAT}
\usepackage[latin1]{inputenc}
\usepackage{amsmath}\usepackage{amsthm}\usepackage{amssymb}\usepackage{mathrsfs}\usepackage{amscd}\usepackage{graphicx}\usepackage{subfigure}\usepackage{amsfonts}\usepackage{amsxtra}\usepackage{color}
\usepackage{enumerate}
\usepackage[left=4cm,right=4cm,top=2cm,bottom=2cm,includeheadfoot]{geometry}

\DeclareMathOperator{\XX}{\mathbb{X}}

\DeclareMathOperator*{\spann}{span}\DeclareMathOperator{\dist}{dist}\DeclareMathOperator{\supp}{supp}

\theoremstyle{definition}\newtheorem{definition}{Definition}[section]\newtheorem{remark}[definition]{Remark}
\theoremstyle{plain}\newtheorem{theorem}[definition]{Theorem}

\newcommand{\Z}{\mathbb{Z}}\newcommand{\R}{\mathbb{R}}

\begin{document}

\begin{frontmatter}

\title{Metric entropy, $n$-widths, and sampling of functions on manifolds}
\author{Martin Ehler\corref{mycorrespondingauthor}}
\address{University of Vienna, Department of Mathematics, Oskar-Morgenstern-Platz 1, A-1090 Vienna
} 
\ead{martin.ehler@univie.ac.at}

\author{Frank Filbir}
\address{Faculty of Mathematics, Technische Universit\"at M\"unchen, 
Boltzmannstrasse 3, 85748, Garching, Germany,  and 
Helmholtz Zentrum M\"unchen,
Ingolst\"adter Landstrasse 1,  
D-85764 Neuherberg
}\ead{filbir@helmholtz-muenchen.de}

%

\begin{abstract}
We first investigate on the asymptotics of the Kolmogorov metric entropy and nonlinear $n$-widths of approximation spaces on some function classes on manifolds and quasi-metric measure spaces. Secondly, we develop constructive algorithms to represent those functions within a prescribed accuracy. The constructions can be based on either spectral information or scattered samples of the target function. Our algorithmic scheme is asymptotically optimal in the sense of nonlinear $n$-widths and asymptotically  optimal up to a logarithmic factor with respect to the metric entropy. 
\end{abstract}


\end{frontmatter}


\section{Introduction}
In classical computational mathematics, it is customary to represent a function by using finitely many parameters, e.g., the coefficients of some truncated series expansion. Representing a function in terms of binary bits rather than a sequence of real numbers is the problem of quantization. The integers thus obtained are represented as a bit string to which coding techniques can be applied to achieve a final compression. In wireless communication, for instance, one needs to transform an analogue signal into a stream of bits, from which the original signal should be recovered at the receiving end with a minimal distortion.

The theory of bit representation of functions pre-dates these modern requirements and was already studied by Kolmogorov. The notion of metric  entropy in the sense of Kolmogorov gives a measurement of the minimal number of bits needed to represent an arbitrary function from a compact subset of a function space. Babenko, Kolmogorov, Tikhomirov, Vitushkin, and Yerokhin have given many estimates on the metric entropy for several compact subsets of the standard function spaces, cf.~\cite{Kolmogorov:1959fk}, \cite[Chapter 2,3]{Vitushkin:1961fk}. 

Constructive algorithms were derived in \cite{Mhaskar:2005fk} to represent functions in suitably defined Besov spaces on the sphere using asymptotically the same number of bits as the metric entropy of these classes, except for a logarithmic factor. A generalization was obtained for compact smooth Riemannian manifolds $\mathbb{X}$ and  global approximation spaces in \cite{Ehler:2013fkxx}. Periodic function classes were considered in \cite{Dung:2001uq,Temlyakov:1990kx}. Related measures of complexity are $n$-widths \cite{Kolmogorov:1936fk} and were studied for some classical function spaces in \cite{Pinkus:1985uq,Pinkus:1985kx}, see \cite{DeVore:1989sp} for the concept of nonlinear $n$-widths. We also refer to \cite[Chapter 1]{Carl:1990fk}, \cite[Sections 1.3, 3, 4]{Edmunds:1996vn}, \cite[Sections 6.3, 6.4]{Haroske:2007ys} for related results.

Both concepts, metric entropy and $n$-widths, are important complexity measures for the analysis of functions on high-dimensional datasets  occurring in biology, medicine, and related areas. Many computational schemes are categorized into the field of manifold learning, where functions need to be learned from finitely many training data that are assumed to lie on some (unknown) manifold \cite{Chui:2015th,Coifman:2005aa,Nadler:2007fk,Roweis:2000aa,Scholkopf:2002fk,Tenenbaum:2000aa}. While much of the recent research in this direction focuses on the understanding of data geometry, approximation theory methods were introduced in \cite{Filbir:2010aa,Filbir:2011fk,Maggioni:2008fk,Mhaskar:2010kx,Mhaskar:2011uq}.

The purpose of the present paper is to generalize results on metric entropy and $n$-widths to the context of sampled functions on quasi-metric measure spaces and covering a larger scale of approximation spaces. Indeed, we  determine the asymptotics of the metric entropy and nonlinear $n$-widths for global approximation spaces. We also provide a computational approximation method, and this scheme is asymptotically optimal with respect to the nonlinear $n$-widths and asymptotically optimal up to a logarithmic factor in the sense of the metric entropy. In addition to obtaining theoretical bounds on the metric entropy and $n$-widths, our results have the following notable features:
\begin{itemize}
\item[-] The computational scheme is based on a linear approximation operator to asymptotically match the optimal bounds in the sense of $n$-widths. 

\item[-] We give explicit schemes for converting the target function into a near minimal number of bits by combining the linear approximation operator with linear quantization, and we derive a reconstruction scheme from such bits to a prescribed accuracy.
\item[-] Our constructions can deal with both, spectral information as well as finitely many training data consisting of function evaluations at scattered data points.
\end{itemize}
In addition, we shall investigate on the asymptotics of the metric entropy of local approximation spaces. 

The outline of this paper is as follows: In Section \ref{sec:bits} we introduce the setting and define metric entropy and $n$-widths. The asymptotics of the metric entropy of global (and under additional assumptions also local) approximation spaces is determined in Section \ref{sec:metric}. In Section \ref{sec:sum kernel}, we introduce our approximation schemes for global approximation spaces and compute the asymptotics of the $n$-widths of global approximation spaces. In Section \ref{sec:wavelet quantization} we verify that linear quantization of the approximation scheme leads to optimal bit representations up to a logarithmic factor for the global approximation space. Local versions of approximation spaces are considered in Section \ref{sec:local}. 
For the readers convenience, \ref{sec assumptions} contains a list and brief discussion of the technical assumptions used for the main results of the present paper. 

\section{Approximation spaces and their metric entropy and $n$-widths}\label{sec:bits}
\subsection{Diffusion measure space}\label{sec:technicalities}
We first fix the setting and introduce some technical assumptions used throughout the paper. 
%
Let $(\mathbb{X},\rho)$ be a quasi-metric space, i.e., a space with a nonnegative symmetric map $\rho$ satisfying $\rho(x,y)=0$ if and only if $x=y$, and the triangle inequality holds at least up to a constant factor $c>0$, i.e., 
\begin{equation*}
\rho(x,y)\leq c(\rho(x,z)+\rho(z,y)), \quad \text{ for all $x,y,z\in\mathbb{X}$.}
\end{equation*}
The quasi-metric $\rho$ induces a topology and we assume that $\mathbb{X}$ is endowed with a Borel probability measure $\mu$. The system $\{\varphi_k\}_{k=0}^\infty\subset L_2(\mathbb{X},\mu)$ is supposed to be an orthonormal basis of continuous real-valued functions with $\varphi_0\equiv 1$ and our results also involve a sequence of nondecreasing real numbers $\{\lambda_k\}_{k=0}^\infty$ such that $\lambda_0=0$ and $\lambda_k\rightarrow\infty$ as $k\rightarrow\infty$. Let $N$ be a positive integer and we shall restrict ourselves to $N=2^n$, where $n$ is some nonnegative integer. The space of \emph{diffusion polynomials} up to degree $N$ is $\Pi_N:=\spann\{\varphi_k:\lambda_k\leq  N \}$, and the \emph{generalized heat kernel} is
\begin{equation}\label{eq:heat}
G_t(x,y)=\sum_{k=0}^\infty \exp(-\lambda^2_k t) \varphi_k(x)\varphi_k(y), \quad t>0.
\end{equation}
We use the symbols $\lesssim$ when the left-hand side is bounded by a generic positive constant times the right-hand side, $\gtrsim$ is used analogously, and $\asymp$ means that both $\lesssim$ and $\gtrsim$ hold. Now, we can summarize the technical assumptions that are related to so-called upper and lower Gaussian bounds on the generalized heat kernel:
\begin{definition}[\cite{Chui:2013fk}]\label{dimmdef}
Under the above notation, a quasi-metric space $\XX$ is called a \emph{diffusion measure space} if there is some $\alpha>0$ such that each of the following properties is satisfied:
\begin{enumerate}[(i)]
\item\label{it: here 1} For each $x\in\XX$ and $t>0$, the closed ball $B_t(x)$ of radius $t$ at $x$  is compact and 
\begin{equation*}
\mu(B_t(x)) \lesssim t^\alpha, \qquad x\in\XX, \ t>0.
\end{equation*}
\item \label{it: here 2}There is $c>0$ such that 
\begin{equation*}
|G_t(x,y)|\lesssim t^{-\alpha/2}\exp\Big(-c\frac{\rho(x,y)^2}{t}\Big), \quad x, y\in \XX,\ 0<t\le 1.
\end{equation*}
\item\label{it: here 3} We have 
$ 
t^{-\alpha/2}\lesssim G_t(x,x), \quad x\in\XX, \ 0<t<1.
$ 
\end{enumerate}
\end{definition}
From here on, we suppose that $\mathbb{X}$ is a diffusion measure space throughout the present paper. It is also noteworthy that the conditions of a diffusion measure space imply that $\mu(B_t(x))\asymp t^\alpha$, for all $0<t<1$, cf.~\cite{Filbir:2010aa}. Thus, the volume of small balls essentially behaves as in $\R^\alpha$. The above conditions also  imply the following estimate on the Christoffel function (or spectral function),
\begin{equation}\label{eq:chris}
\sum_{\lambda_k\leq t} |\varphi_k(x)|^2 \asymp t^\alpha,\quad x\in\mathbb{X}, \; t\geq 1,
\end{equation}
see \cite{Chui:2013fk,Filbir:2010aa,Filbir:2011fk} for a discussion and references. By integrating over $\mathbb{X}$, we obtain 
\begin{equation}\label{eq:dim formel und so}
\dim(\Pi_t)\asymp t^\alpha.
\end{equation}



\begin{remark}\label{rem 2.2}
It was pointed out in \cite{Chui:2013fk} that all technical assumptions are satisfied when $\mathbb{X}\subset\R^d$ is an $\alpha$-dimensional compact, connected, Riemannian manifold without boundary, with non-negative Ricci curvature, geodesic distance $\rho$, and $\mu$ being the Riemannian volume measure on $\mathbb{X}$ normalized with $\mu(\mathbb{X})=1$, $\{\varphi_k\}_{k=0}^\infty$ are the eigenfunctions of the Laplace-Beltrami operator on $\mathbb{X}$, and $\{-\lambda_k\}_{k=0}^\infty$ are the corresponding eigenvalues arranged in nonincreasing order, see also \cite{Grigoryan:1999fk}. In this case, \eqref{eq:dim formel und so} is consistent with Weyl's law. 
For further discussions, we refer to \cite{Filbir:2010aa,Filbir:2011fk,Mhaskar:2010kx}.
\end{remark}
Given an arbitrary normed space $Z$ and a subset $Y\subset Z$, we define, for $f\in Z$,
\begin{equation}\label{eq:best error}
E(f,Y,Z) : = \inf_{g\in Y} \|f-g\|_Z.
\end{equation}
As we have pointed out already, we assume $\mathbb{X}$ is a diffusion measure space throughout. The notation $|B$ in the following means restricting functions to the set $B$. 
\begin{definition}\label{def:sob first one}
For a nonempty ball $B\subset\mathbb{X}$ and $1\leq p\leq \infty$, the \emph{approximation space} of order $s>0$ is defined by 
\begin{equation}\label{eq:def of Sobolev}
\mathcal{A}^s(L_p(B)) = \{f\in L_p(B) : \|f\|_{\mathcal{A}^s(L_p(B))}<\infty \},
\end{equation}
where the associated norm is given by 
\begin{equation*}
\|f\|_{\mathcal{A}^s(L_p(B))}:=\|f\|_{L_p(B)} + \sup_{N\geq 1} N^sE(f,\Pi_{N|B},L_p(B)).
\end{equation*} 
The unit ball in $\mathcal{A}^s(L_p(B))$ is denoted by 
\begin{equation}\label{eq:unit ball Sobolev}
\overline{\mathcal{A}^s}(L_p(B)):=\{f\in L_p(B): \|f\|_{\mathcal{A}^s(L_p(B))} \leq 1  \}.
\end{equation}
\end{definition}
For a nonempty ball $B\subset\mathbb{X}$, let us denote the closure of $\bigcup_N\Pi_{N|B}$ in $L_p(B)$ by $X_p(B)$. We can simply switch from $L_p(B)$ to $X_p(B)$ in the definition of approximation space without changing anything, so that we observe $\mathcal{A}^s(L_p(B)) = \mathcal{A}^s(X_p(B))$.

\begin{remark}
In classical situations, the smoothness of a function is related to the accuracy of its approximation, and for a more classical characterization of the approximation spaces in terms of pseudo-differential operators and K-functionals, we refer to \cite{Maggioni:2008fk}. In turn it has also become customary to consider the accuracy of approximation itself as a measurement of smoothness. For classical results on approximation spaces, we refer to \cite[Chapter 7]{DeVore:1993ab} and references therein. 
\end{remark}



\subsection{Kolmogorov metric entropy and nonlinear $n$-widths}\label{sec:metric entropy sobolev}
Metric entropy as studied in \cite{Lorentz:1966fk} refers to the minimal number of bits needed to represent a function $f$ up to precision $\varepsilon$. This number determines the maximal compression when loss of information is bounded by $\varepsilon$. For a more stringent mathematical exposition, let $Y$ be a compact subset of a metric space $(Z,\varrho)$. Given $\varepsilon>0$, let $\mathcal{N}_\varepsilon(Y)$ be the \emph{$\varepsilon$-covering number of $Y$ in $Z$}, i.e., the minimal number of balls of radius $\varepsilon$ that cover $Y$. Suppose that $g_1,\ldots,g_{\mathcal{N}_\varepsilon(Y)}$ is a list of centers of these balls. Given any $f\in Y$, there is $g_j$ such that $\varrho(f,g_j)\leq \varepsilon$. We may then represent $f$ using the binary representation of $j$, and use $g_j$ as the reconstruction of $f$ based on this representation. Any binary enumeration of these centers takes $\lceil\log_2(\mathcal{N}_\varepsilon(Y))\rceil$ many bits, which somewhat measures the complexity of $Y$ and where $\lceil \cdot \rceil$ denotes the ceiling function. 
\begin{definition}
Let $Y$ be a compact subset of a metric space $(Z,\varrho)$ and, for $\varepsilon>0$, let $\mathcal{N}_\varepsilon(Y)$ be the $\varepsilon$-covering number of $Y$ in $Z$. Then 
\begin{equation}\label{eq:def entropy}
H_\varepsilon(Y,Z):=\log_2(\mathcal{N}_\varepsilon(Y))
\end{equation}
is called the \emph{metric entropy} of $Y$ in $Z$.
\end{definition}
Thus, the smallest integer not less than the metric entropy is the minimal number of bits necessary to represent any $f$ with precision $\varepsilon$. 

Let us also introduce some alternative notions of complexity. Given some Banach space $Z$, let $n\geq 1$ be an integer, and let $M_n:\R^n\rightarrow Z$ be some mapping, for which $Z_n:=M_n(\R^n)$ denotes its range.  For a compact subset $Y\subset Z$, we define the worst case error of approximation, consistently with \eqref{eq:best error}, by 
\begin{equation*}
E(Y,Z_n,Z) : = \sup_{f\in Y}E(f,Z_n,Z).
\end{equation*}
Endow $Y$ with some quasi-norm and consider continuous maps $\bar{a}:Y\rightarrow \R^n$. For fixed $\bar{a}$, the term $M_n(\bar{a}(f))\in Z$ is an approximation of $f$ from $Z_n$. The quantity 
\begin{equation*}
E(Y,\bar{a},M_n,Z) := \sup_{f\in Y} \|f-M_n(\bar{a}(f))\|_Z
\end{equation*}
is the error of approximation by the nonlinear method of approximation $M_n(\bar{a}(\cdot)):Y\rightarrow Z$. The \emph{continuous nonlinear $n$-width} is defined in \cite{DeVore:1989sp} by 
\begin{equation*}
d_n(Y,Z): = \inf _{\bar{a},M_n} E(Y,\bar{a},M_n,Z), 
\end{equation*}
where the infimum runs over all continuous maps $\bar{a}:Y\rightarrow \R^n$ and all mappings $M_n:\R^n\rightarrow Z$. For further considerations on $n$-widths, we refer to \cite{DeVore:1989sp,Dung:2013uq,Foucart:2010fk,Micchelli:1977kx,Nowak:1995ys,Traub:1980vn}. 

In the subsequent sections we shall compute the metric entropy \eqref{eq:def entropy} and the continuous nonlinear $n$-widths of the approximation ball $\overline{\mathcal{A}^s}(X_p(\mathbb{X}))$ of radius $1$ given by \eqref{eq:unit ball Sobolev}. We refer to \cite{Geller:2012fk} for related results.

\subsection{The metric entropy of approximation spaces}\label{sec:metric}
We shall determine the asymptotics of the metric entropy of global and local approximation spaces, so let $B\subset \mathbb{X}$ be some nonempty ball, which is allowed to coincide with $\mathbb{X}$. Since $\mathcal{A}^s(X_p(B))$ is not finite-dimensional, $\overline{\mathcal{A}^s}(X_p(B))$ is not compact in the approximation space. Here, we consider $\mathcal{A}^s(X_p(B))$ as a subspace of $X_p(B)$, in which it is compact, see \cite{Ehler:2013fkxx} for $B=\mathbb{X}$ and the case $B\subsetneq \mathbb{X}$  can be proven analogously.  

The following result for the approximation space extends findings in \cite{Mhaskar:2005fk} from the sphere to diffusion measure spaces. 
\begin{theorem}\label{th:op bit}
Let $\mathbb{X}$ be a diffusion measure space and suppose we are given a nonempty ball $B\subset\mathbb{X}$ such that $\{\varphi_{k|B}\}_{j=0}^\infty$ are linearly independent. If $s>0$ is fixed, then 
\begin{equation}\label{eq:entropy 0}
H_\varepsilon(\overline{\mathcal{A}^s}(X_p(B)),X_p(B)) \asymp  (1/\varepsilon)^{\alpha/s}, \quad \text{for all } 0<\varepsilon\leq 1, 
\end{equation} 
 where $\alpha$ is the constant in Definition \ref{dimmdef}.
\end{theorem}
This result was verified in \cite{Ehler:2013fkxx} provided that $B=\mathbb{X}$ and $p=\infty$. Note that the linear independence condition is trivially satisfied for $B=\mathbb{X}$ since $\{\varphi_{k}\}_{j=0}^\infty$ is even an orthonormal basis. In case $B$ is a proper subset of $\mathbb{X}$, the linear independence condition may be hard to check in general. However, if $\mathbb{X}$ is a compact, connected, real-analytic Riemannian manifold (equipped with a real analytic Riemannian metric), then the eigenfunctions $\{\varphi_{k}\}_{j=0}^\infty$ of the Laplace-Beltrami operator are real analytic due to the real analytic hypoellipticity of elliptic partial differential operators, cf.~\cite[Section 5.3]{Krantz:1992yg}. In this situation, the global linear independence implies the linear independence of the restrictions. For results related to Theorem \ref{th:op bit}, see \cite{Dung:2013uq,Nowak:1995ys,Pinkus:1985uq}. 

The proof of Theorem \ref{th:op bit} is based on a general Banach space result, also used in \cite[Theorem 4.1]{Mhaskar:2005fk}. Let $Z$ be a Banach space and $\{\phi_k\}_{k=1}^\infty\subset Z$ be a sequence of linearly independent elements whose linear span is dense in $Z$, and define $Z_k:=\spann\{\phi_1,\ldots,\phi_k\}$ with $Z_0=\{0\}$. Let $\{\delta_k\}_{k=0}^\infty$ be a nonincreasing sequence of positive numbers with $\lim_{k\rightarrow\infty}\delta_k=0$ and define
\begin{equation}\label{eq:approx space}
\mathcal{A}(Z;\{\delta_k\}_{k=0}^\infty,\{\phi_k\}_{k=1}^\infty): = \{f\in Z : E(f,Z_k,Z)\leq \delta_k,\;\text{for } k=0,1,\ldots\}.
\end{equation}
The following result goes back to Lorentz in \cite{Lorentz:1966fk}:
\begin{theorem}[Theorem 3.3 in \cite{Lorentz:1996uq}]\label{th:Lorentz}
Let $\{\delta_k\}_{k=0}^\infty$ be a nonincreasing sequence of positive numbers such that $\delta_{2k}\leq c \delta_k$, for $k=1,2,\ldots$ and some constant $c\in (0,1)$. For $\ell\geq 0$, let $M_\ell:=\min\{k : \delta_k\leq e^{-\ell}\}$, then we have, for $0<\varepsilon\leq 1$,
\begin{equation}\label{eq:Heps}
H_\epsilon\big(\mathcal{A}(Z;\{\delta_k\}_{k=0}^\infty,\{\phi_k\}_{k=1}^\infty),Z  \big)  \asymp \sum_{\ell=1}^L M_\ell,
\end{equation}
where $L:=2+\lfloor \log(1/\varepsilon) \rfloor$.
\end{theorem}
\begin{proof}[Proof of Theorem \ref{th:op bit}]
We aim to apply Theorem \ref{th:Lorentz} with the function system $\{{\varphi_{k|B}}\}_{k=0}^\infty$ and $Z=X_p(B)$. There, the index set is supposed to start with $k=1$, so we set $\phi_k={\varphi_{k-1|B}}$, $k=1,2,\ldots$. To define the sequence $\{\delta_k\}_{k=0}^\infty$, we need some preparation. The linear independence assumption yields that \eqref{eq:dim formel und so} implies $\dim(\Pi_{N|B})\asymp N^\alpha$. 
By using $Z_k:=\spann\{\varphi_{0|B},\ldots,\varphi_{k-1|B}\}$, we derive, for $N^\alpha\leq k\leq (2N)^\alpha$,
\begin{equation*}
(2N)^s E(f,\Pi_{2N|B},X_p(B)) \lesssim k^{s/\alpha}E(f,Z_k,X_p(B))  \lesssim N^s E(f,\Pi_{N|B},X_p(B)).
\end{equation*}
Therefore, there are constants $C_i\geq 1$, for $i=1,2$, such that the definitions $\delta_{1;0}=\frac{1}{2}$, $\delta_{1;k} := (2C_1)^{-1} k^{-s/\alpha}$, and $\delta_{2;0}=C_2$, $\delta_{2;k} := C_2 k^{-s/\alpha}$, lead to
\begin{equation*}
\mathcal{A}(X_p(B);\{\delta_{1;k}\}_{k=0}^\infty,\{\phi_k\}_{k=1}^\infty) \subset \overline{\mathcal{A}^s}(X_p(B))\subset \mathcal{A}(X_p(B);\{\delta_{2;k}\}_{k=0}^\infty,\{\phi_k\}_{k=1}^\infty),
\end{equation*}
which also yields
\begin{align*}
\mathcal{N}_\varepsilon\big(\mathcal{A}(X_p(B);\{\delta_{1;k}\}_{k=0}^\infty,\{\phi_k\}_{k=1}^\infty)\big) & \leq  \mathcal{N}_\varepsilon(\overline{\mathcal{A}^s}(X_p(B))) \\&\leq \mathcal{N}_\varepsilon\big(\mathcal{A}(X_p(B);\{\delta_{2;k}\}_{k=0}^\infty,\{\phi_k\}_{k=1}^\infty)\big).
\end{align*}
 Since $\delta_{i;2k}\leq c\delta_{i;k}$, for $c:=2^{-s/\alpha}\in(0,1)$, we can apply Theorem \ref{th:Lorentz}. According to \cite[Lemma 4.1]{Mhaskar:2005fk}, $\sum_{\ell=1}^L M_\ell  \asymp e^{L\alpha/s}$, so that the choice of $L$ in \eqref{eq:Heps} implies \eqref{eq:entropy 0}. 
\end{proof}

%
%

%
%
%
%
%
%
%
%

It is obvious that increased precision requires more bits, and smoother functions can be represented with fewer bits. The exact growth condition \eqref{eq:entropy 0} reflects these thoughts in a quantitative fashion and that \eqref{eq:entropy 0} serves as a benchmark for function representation on diffusion measure spaces. The remaining part of the present work is dedicated to develop a scheme that matches the optimality bound at least up to a logarithmic factor.

\section{Global function approximation}\label{sec:sum kernel}
This section is dedicated to introduce our approximation scheme, to discuss its characterization of certain function spaces, and to determine the asymptotics of the $n$-widths. 

\subsection{Localized kernels}
We now collect few ingredients for our approximation scheme.
\begin{definition}\label{def:filter}
We call an infinitely often differentiable and non-increasing function $H:\R_{\geq 0}\rightarrow\R$ a \emph{low-pass filter} if $H(t)=1$ for $t\le 1/2$ and $H(t)=0$ for $t\ge 1$.
\end{definition}
For any low-pass filter $H$, the kernel 
\begin{equation}\label{eq:kernel def}
K_N(x,y) : = \sum_{k=0}^\infty H(\frac{\lambda_k}{N})\varphi_k(x)\varphi_k(y),
\end{equation}
cf.~\cite{Filbir:2010aa,Maggioni:2008fk,Mhaskar:2010kx}, is localized, i.e., for fixed $r>\alpha$ and all $x\neq y$ with $N=1,2,\ldots$,
\begin{equation}\label{eq:localization result!}
\big| K_N(x,y) \big| \lesssim  \frac{N^{\alpha-r}}{\rho(x,y)^r}.\end{equation} 
Alternatively, we also have for fixed $r>\alpha$ and all $x,y$ with $N=1,2,\ldots$,
\begin{equation}\label{eq:localization result!2}
\big| K_N(x,y) \big| \lesssim  \frac{N^{\alpha}}{\max(1,(N^r\rho(x,y)^r)}.
\end{equation}
We find in \cite[Inequality (3.12)]{Filbir:2010aa} that 
\begin{equation}\label{eq:boundedness by 1}
\sup_{y\in\mathbb{X}} \int_{\mathbb{X}} \big|K_N(x,y) \big|d\mu(x) \lesssim 1
\end{equation}
holds. 

For some signed Borel measure $\nu$ on $\mathbb{X}$ and $f\in L_1(\mathbb{X},|\nu|)$, we can define, for $N=1,2,\ldots$,
\begin{equation}\label{eq:sigma again and new}
\sigma_{N}(f,\nu) : = \sum_{k=0}^\infty H(\frac{\lambda_k}{N})\int_{\mathbb{X}} f(x) \varphi_k(x)\varphi_k d\nu(x)= \int_{\mathbb{X}} f(x) K_N(x,\cdot) d\nu(x).
\end{equation}

\subsection{Quadrature measures}
This section is dedicated to introduce further ingredients to develop our approximation scheme. We first aim to replace the integral over diffusion polynomials with a finite sum or at least with an integral over a ``simpler'' measure. 
\begin{definition}\label{def:quad}
 A signed Borel measure $\nu$ on $\mathbb{X}$ is called a \emph{quadrature measure} of order $N$ if 
\begin{equation*}
\int_{\mathbb{X}} f(x)g(x)d\mu(x) = \int_{\mathbb{X}} f(x)g(x)d\nu(x),\qquad\text{for all } f,g\in \Pi_{N}.
\end{equation*}
\end{definition}
While $\Pi_N$ is finite dimensional, the collection of products is so as well and, hence, there do exist finitely supported quadrature measure, cf.~\cite{Harpe:2005fk}, \cite[Section 1.5]{Rivlin:1974if}. Note that we request exact integration of products from functions in $\Pi_N$, which leads us to the following product assumption on the structure of diffusion polynomials, which is also used in \cite{Filbir:2011fk}: for $f\in L_p(\mathbb{X})$, let 
\begin{equation*}
\dist(f,\Pi_{N})_{L_\infty}:=\inf_{h\in\Pi_N} \|f-h\|_{L_\infty}
\end{equation*}
and assume that there is a constant $a\geq 2$ such that the quantity 
\begin{equation}\label{eq:def of eps}
\epsilon_N : = \sup_{\lambda_\ell,\lambda_k\leq N} \dist(\varphi_\ell \varphi_k,\Pi_{aN})_{L_\infty}
\end{equation}
satisfies $N^m\epsilon_N \rightarrow 0$ as $N\rightarrow \infty$, for all $m>0$. 

\begin{definition}\label{def:MZ}
Let the product assumption hold. For fixed $1\leq p\leq \infty$, a family $(\mu_N)_{N=1}^\infty$ of positive quadrature measures of order $N$, respectively,  is called a family of \emph{Marcinkiewicz-Zygmund quadrature measures} of order $N$, respectively, if
\begin{equation}\label{MZ-estimate 2}
\|f\|_{|\mu_N|,L_p(\mathbb{X})} \asymp \|f\|_{L_p(\mathbb{X})}, \quad\text{ for all $f\in\Pi_{N}$,}
\end{equation}
where $|\mu_N|$ denotes the total variation measure of $\mu_N$ and  $\|f\|_{|\mu_N|,L_p(\mathbb{X})}$ denotes the $L_p$-norm of $f$ with respect to $|\mu_N|$.
%
%
%
\end{definition}



Under fairly general assumptions, the results in \cite[Theorem 3.1]{Filbir:2010aa} and \cite[Theorem 5.8]{Filbir:2011fk} imply the existence of a family of  finitely supported Marcinkiewicz-Zygmund quadrature measures $(\mu_N)_{N=1}^\infty$ of order $N$, respectively,  such that $\#\supp(\mu_N) \asymp N^\alpha$, where $\alpha$ is as in Definition \ref{dimmdef}.


\subsection{Widths and characterization of approximation spaces}

Fix $1\leq p\leq\infty$ and suppose that $(\mu_N)_{N=1}^\infty$ is a family of Marcinkiewicz-Zygmund quadrature measures of order $N$, respectively. According to \cite[Inequality (3.13)]{Filbir:2010aa}, we have 
\begin{equation}\label{eq:sigma Lp}
 \|\sigma_N(f,\mu_N)\|_{L_p} \lesssim \|f\|_{|\mu_N|,L_p},\qquad f\in L_p(\mathbb{X},|\mu_N|),
\end{equation}
and the generic constant can be chosen independently of $f$ and $N$. Later, we shall also need that \eqref{MZ-estimate 2} extends to all functions in $X_p(\mathbb{X})$, i.e., that $\|f\|_{|\mu_N|,L_p}\lesssim \|f\|_{L_p(\mathbb{X})}$ holds, for all $f\in X_p(\mathbb{X})$. The latter is fine for $p=\infty$. For $1\leq p<\infty$, we have not yet found any explicit example except for the measure $\mu$ itself. Therefore, we shall simply restrict ourselves to $\mu_N=\mu$, $N=1,2,\ldots$ in this case.

We can also estimate 
\begin{equation}\label{eq:one estimate}
\sup_{y\in\mathbb{X}} \int_{\mathbb{X}}|K_N(x,y)| d|\mu_N|(x) \lesssim 1,
\end{equation}
cf.~\cite[Inequality (3.12)]{Filbir:2010aa}. Note that \eqref{eq:one estimate} is the quadrature version of \eqref{eq:boundedness by 1}. 
Next, we recall the characterization of $\mathcal{A}^s(L_p(\mathbb{X}))$ using $\sigma_N$, see \cite{Filbir:2010aa,Maggioni:2008fk,Mhaskar:2010kx,Mhaskar:2011uq}:
\begin{theorem}\label{th:smoothness}
Suppose that $1\leq p\leq \infty$ and assume that $(\mu_N)_{N=1}^\infty$ is a family of Marcinkiewicz-Zygmund quadrature measures of order $N$, respectively, if $p=\infty$. For $1\leq p<\infty$ we choose $\mu_N=\mu$, $N=1,2,\ldots$. Assume further that $H$ is a low-pass filter. Then, for all $f\in \mathcal{A}^s(L_p(\mathbb{X}))$, we have 
\begin{equation}\label{eq:smooth charact}
\|f-\sigma_N(f,\mu_N)\|_{L_p(\mathbb{X})} \lesssim N^{-s}\|f\|_{\mathcal{A}^s(L_p(\mathbb{X}))},
\end{equation}
where the generic constant does not depend on $N$ or $f$. On the other hand, if, for $f\in L_p(\mathbb{X})$, there are generic constants not depending on $N$ such that 
\begin{equation*}
\|f-\sigma_N(f,\mu_N)\|_{L_p(\mathbb{X})} \lesssim N^{-s},
\end{equation*}
then $f\in \mathcal{A}^s(L_p(\mathbb{X}))$.
\end{theorem}

\begin{remark}
Let  us point out again that we suppose $\mu_N=\mu$, $N=1,2,\ldots$  for  $1\leq p<\infty$. In this case, the term $\sigma_N(f,\mu_N)$ contains spectral information  $\int_{\mathbb{X}} f(x) \varphi_k(x)d\mu(x)$. If $p=\infty$ and $\mu_N$ has finite support, then we have an approximation scheme that uses finitely many training data consisting of function evaluations at scattered data points in $\supp(\mu_N)$.
\end{remark}

We have already determined the asymptotics of the Kolmogorov metric entropy. Here, we shall determine the $n$-widths for the global approximation space. 
\begin{theorem}\label{th:nwidth}
The continuous nonlinear $n$-widths of $\overline{\mathcal{A}^s}(X_p(\mathbb{X}))$ in $X_p(\mathbb{X})$ satisfies 
\begin{equation*}
d_n(\overline{\mathcal{A}^s}(X_p(\mathbb{X})),X_p(\mathbb{X})) \asymp n^{-s/\alpha}.
\end{equation*}
\end{theorem}
In order to verify Theorem \ref{th:nwidth} we shall consider two more types of $n$-widths, cf.~\cite{Pinkus:1985kx}. 
The \emph{linear $n$-width of a subset $Y$ in a Banach space $Z$} is
\begin{equation*}
\mathscr{L}_n(Y,Z):=\inf_{F_n} \sup_{x\in Y}\|x-F_n(x)\|,
\end{equation*}
where the infimum is taken over all bounded linear operators $F_n$ on $Z$ whose range is of dimension at most $n$. The \emph{Bernstein $n$-width of $Y$ in $Z$} is
\begin{equation*}
\mathscr{B}_n(Y,Z):=\sup_{Z_{n+1}} \sup\{\lambda: \lambda \overline{Z_{n+1}}\subset Y\}
\end{equation*}
where the supremum is taken over all subspaces $Z_{n+1}$ of $Z$ of dimension at least $n+1$ and $\overline{Z_{n+1}}$ denotes the unit ball in $Z_{n+1}$. The continuous nonlinear $n$-widths is sandwiched by 
\begin{equation}\label{eq:sandwich}
\mathscr{B}_n(Y,Z) \leq d_n(Y,Z)\leq \mathscr{L}_n(Y,Z),
\end{equation}
cf.~\cite{DeVore:1989sp}. 
We can now take care of the proof. We point out that we shall derive a lower bound on the Bernstein $n$-width, which is closely related to the Bernstein estimates in \cite{Maggioni:2008fk} for integer derivatives when combined with K-functionals, cf.~\cite{DeVore:1989sp}. 
\begin{proof}[Proof of Theorem \ref{th:nwidth}]
The operator $\sigma_N$ is bounded on $L_p(\mathbb{X})$ and $\sigma_N(f)$ is an element in $\Pi_N$. Consider $N$ with $N\asymp n^{1/\alpha}$, so that $\dim(\Pi_N)\asymp N^\alpha$ yields
\begin{equation*}
\mathscr{L}_n(\overline{\mathcal{A}^s}(X_p(\mathbb{X})),X_p(\mathbb{X}))\lesssim n^{-s/\alpha},
\end{equation*}
cf.~Theorem \ref{th:smoothness} . Hence, the second inequality in \eqref{eq:sandwich} implies $d_n(\overline{\mathcal{A}^s}(X_p(\mathbb{X})),X_p(\mathbb{X}))\lesssim n^{-s/\alpha}$.

To establish the lower bound, we take the subspace $\Pi_{N+1}$ and aim to derive a generic constant $c>0$ such that $cN^{-s}\overline{\Pi_{N+1}}\subset  \overline{\mathcal{A}^s}(X_p(\mathbb{X}))$, where $\overline{\Pi_{N+1}}=\{f\in\Pi_{N+1}: \|f\|_{L_p(\mathbb{X})}\leq 1\}$. For $f\in\overline{\Pi_{N+1}}$, let $g:=N^{-s}f$. We obtain $\|g\|_{L_p(\mathbb{X})}\leq N^{-s}$ and, for $M>N$, we derive $E(g,\Pi_M,L_p(\mathbb{X}))=0$. The choice $M\leq N$ yields
\begin{align*}
E(g,\Pi_M,L_p(\mathbb{X})) = N^{-s} \|f-\sigma_M(f)\|_{L_p(\mathbb{X})}
 \lesssim N^{-s},
\end{align*}
because $\|f\|_{L_p(\mathbb{X})}\leq 1$ and $\|\sigma_M(f)\|_{L_p(\mathbb{X})}\lesssim \|f\|_{L_p(\mathbb{X})}$. 
Thus, 
\begin{equation*}
\sup_{M\geq 1} M^sE(g,\Pi_M,L_p(\mathbb{X})) \lesssim 1
\end{equation*}
 holds. We have verified that there is a generic constant $c$ such that 
 \begin{equation*}
 cN^{-s}\overline{\Pi_{N+1}} \subset \overline{\mathcal{A}^s}(L_p(\mathbb{X})).
 \end{equation*}
 Therefore, we obtain $\mathscr{B}_n(\overline{\mathcal{A}^s}(X_p(\mathbb{X})),X_p(\mathbb{X}))\gtrsim n^{-s/\alpha}$, so that \eqref{eq:sandwich} concludes the proof.
\end{proof}
It should be mentioned that upper bounds on linear $n$-widths  for compact Riemannian manifolds were already derived in \cite{Geller:2013uk}, where also the exact asymptotics were obtained for compact homogeneous manifolds. 

\section{Bit representation in global approximation spaces} \label{sec:wavelet quantization}
This section is dedicated to verify that linear quantization of the approximation scheme $\sigma_N(f,\mu_N)$ enables bit representations matching the optimality bounds stated in Theorem \ref{th:op bit} up to a logarithmic factor. First, we recall the formula \eqref{eq:sigma again and new},
\begin{equation*}
\sigma_N(f,\mu_N)=
\sum_{k=0}^\infty H(\frac{\lambda_k}{N})\int_{\mathbb{X}} f(x) \varphi_k(x)\varphi_kd\mu_N(x),
\end{equation*}
where $(\mu_N)_{N=1}^\infty$ is a family of Marcinkiewicz-Zygmund quadrature measures of order $N$, respectively, if $p=\infty$. Again, if $1\leq p<\infty$, then we choose $\mu_N=\mu$, $N=1,2,\ldots$. 

Since $H(t)=1$, for $t\in[0,1/2]$ and $H(t)=0$, for $t>1$, we observe that $H(\frac{\lambda_k}{N})H(\frac{\lambda_k}{2N})=H(\frac{\lambda_k}{N})$. If $(\nu_N)_{N=1}^\infty$ is a family of quadrature measures of order $2N$, respectively, then a straight-forward calculation yields
\begin{equation}\label{eq:for later use}
\sigma_N(f,\mu_N) = \int_{\mathbb{X}}\sigma_N(f,\mu_N,x)\sum_{k=0}^\infty H(\frac{\lambda_k}{2N})\varphi_k(x)d\nu_N(x)\varphi_k,
\end{equation}
The representation \eqref{eq:for later use} involves the quadrature measure $\nu_N$ and the Marcinkiewicz-Zygmund quadrature measure $\mu_N$. To design the final approximation scheme, we fix some $S>1$, apply the quantization
\begin{equation}\label{eq:quantization2}
I_N(f,\mu_N,x)=\lfloor N^{S}\sigma_N (f,\mu_N,x) \rfloor,
\end{equation}
 and define the actual approximation by 
\begin{equation}\label{eq:discretized approx operator2}
\sigma^\circ_N(f,\mu_N,\nu_N):= N^{-S} \int_{\mathbb{X}} I_N(f,\mu_N,x) \sum_{k=0}^\infty H(\frac{\lambda_k}{2N})\varphi_k(x)d\nu_N(x)\varphi_k.
\end{equation}
In other words, we replace $\sigma_N (f,\mu_N,x)$ in \eqref{eq:for later use} with a number on the grid $\frac{1}{N^S}\Z $. 

We have the following result for the ball $\overline{\mathcal{A}^s}(L_p(\mathbb{X}))$ of radius $1$ of the global approximation space given by \eqref{eq:def of Sobolev}. It extends results in \cite{Ehler:2013fkxx} from compact Riemannian manifolds and $p=\infty$ to diffusion measure spaces and to the entire range $1\leq p\leq\infty$:
\begin{theorem}\label{th:alles22}
For the case $p=\infty$, we assume that $(\mu_N)_{N=1}^\infty$ is a family of Marcinkiewicz-Zygmund quadrature measures of order $N$, respectively. For $1\leq p<\infty$ we choose $\mu_N=\mu$, $N=1,2,\ldots$. Assume further that $H$ is a low-pass filter. We suppose that $(\nu_N)_{N=1}^\infty$ are Marcinkiewicz-Zygmund quadrature measures of order $2N$ with $\#\supp(\nu_N) \lesssim N^\alpha$. For fixed $s>0$ and $S>\max(1,s)$, we apply the discretizations \eqref{eq:quantization2} and \eqref{eq:discretized approx operator2}. Then there is a constant $c>0$ such that, for all $f\in \overline{\mathcal{A}^s}(L_p,\mathbb{X})$,
\begin{equation}\label{eq:approx inequality in theorem2}
\|f-\sigma^\circ_N(f,\mu_N,\nu_N)\|_{L_p(\mathbb{X})} \leq c  N^{-s}
\end{equation}
holds. For $cN^{-s} = \varepsilon \leq 1$ and $\varepsilon\leq 1$, the number of bits needed to represent all integers $\{I_N(f,\mu_N,x): x\in  \supp(\nu_N) \}$ does not exceed a positive constant (independent of $\varepsilon$) times 
\begin{equation}\label{eq:almost optimal2}
(1/\varepsilon)^{\alpha/s} (1+\log_2(1/\varepsilon) ).
\end{equation}
\end{theorem}
\begin{proof}[Proof of Theorem \ref{th:alles22}]
The triangle inequality yields 
\begin{align*}
\|f-\sigma^\circ_N(f,\mu_N,\nu_N)\|_{L_p(\mathbb{X})} &\lesssim  \|f-\sigma_N(f,\mu_N)\|_{L_p(\mathbb{X})} + \|\sigma_N(f,\mu_N)-\sigma^\circ_N(f,\mu_N,\nu_N) \|_{L_p(\mathbb{X})} .
\end{align*}
Since Theorem \ref{th:smoothness} implies $\|f-\sigma_N(f,\mu_N)\|_{L_p(\mathbb{X})}\lesssim N^{-s}\|f\|_{\mathcal{A}^s(L_p(\mathbb{X}))}$, we only need to take care of the term on the far most right. The quantization \eqref{eq:quantization2} immediately yields 
\begin{equation}\label{eq:quantization immediately yields2}
|\sigma_N(f,\mu_N,x)-N^{-S}I_N(f,\mu_N,x)|\leq N^{-S},\quad\text{for all $x\in\supp(\nu_N)$,}
\end{equation}
so that \eqref{eq:for later use} and \eqref{eq:one estimate} imply
\begin{align*}
\|\sigma_N(f,\mu_N)-\sigma^\circ_N(f,\mu_N,\nu_N)\|_{L_p(\mathbb{X})} & \\
&\hspace{-1.5cm}= \big\|\int_\mathbb{X} \Big( \sigma_N(f,\mu_N,y) -N^{-S}I_N(f,\mu_N,y)\Big)K_{2N}(\cdot,y)d\nu_N(y) \big\|_{L_p(\mathbb{X})} \\
&\hspace{-1.5cm} \lesssim N^{-S}\leq N^{-s}.
\end{align*}
Hence, we have derived \eqref{eq:approx inequality in theorem2}.

To tackle \eqref{eq:almost optimal2}, we observe that the localization property \eqref{eq:localization result!2} yields $\|g\|_{L_\infty}\lesssim N^\alpha\|g\|_{L_1}$, for all $g\in\Pi_{N}$, see also \cite[Lemma 5.5]{Mhaskar:2010kx} for more general Nikolskii inequalities. We apply \eqref{eq:quantization immediately yields2} and then use $\sigma_N(f,\mu_N)\in\Pi_N$ with $L_p \hookrightarrow L_1$,
which yields 
\begin{equation*}
| I_N(f,\mu_N,x) | \lesssim N^S \| \sigma_N(f,\mu_N) \|_{L_\infty(\mathbb{X})}  \lesssim N^{S+\alpha}\|\sigma_N(f,\mu_N) \|_{L_p(\mathbb{X})}.
\end{equation*}
According to \cite[Theorem 5.1]{Mhaskar:2010kx}, $\|\sigma_N(f,\mu_N) \|_{L_p(\mathbb{X})}\lesssim \|f\|_{L_p(\mathbb{X})}$ holds. Since $f$ is contained in the ball of radius $1$, so that $\|f\|_{L_p(\mathbb{X})}\leq 1$, we see that 
\begin{equation*}
| I_N(f,\mu_N,x) |\lesssim  N^{S+\alpha}.
\end{equation*}
Thus, the number of bits needed to represent each $I_N(f,\mu_N,x)$ is at most $\log_2(c_1N^{S+\alpha})$, where $c_1\geq 1$ is a positive constant. Note that we can assume that $c_1N^{S+\alpha}\geq 1$ because otherwise $ I_N(f,\mu_N,x)$ would be zero. 
Since  $\# \supp(\nu_N) \lesssim N^\alpha$, we have $\#\{I_N(f,\mu_N,x) : x\in\supp(\nu_N)\}\lesssim N^\alpha$. Therefore, the total number of bits needed to represent all numbers $\{I_N(f,\mu_N,x) : x\in\supp(\nu_N)\}$ is at most $c_2 N^\alpha\log_2(c_1N^{S+\alpha})$, where $c_2$ is a positive constant. By using $cN^{-s}=\varepsilon\leq 1$ and $\varepsilon\leq 1$, we derive that the number of necessary bits does not exceed 
\begin{align*}
c_2 c^{\alpha/s}(1/\varepsilon)^{\alpha/s} \log_2(c_1 (c/\varepsilon)^{(S+\alpha)/s} ) &\lesssim 
(1/\varepsilon)^{\alpha/s} \log_2((c_1)^{s/(S+\alpha)}c/\varepsilon)\\
& \lesssim (1/\varepsilon)^{\alpha/s} \log_2((c_1/\varepsilon)^{s/(S+\alpha)}c/\varepsilon)\\
&\lesssim (1/\varepsilon)^{\alpha/s} (1+\log_2(1/\varepsilon)),
\end{align*}
which concludes the proof. 
\end{proof}
Theorem \ref{th:alles22} yields that our bit representation scheme is optimal with respect to the metric entropy as stated in Theorem \ref{th:op bit} at least up to a logarithmic factor. 

\section{Bit representation of locally smooth functions}\label{sec:local}
In the previous section we derived a bit-representation scheme for the global approximation space, i.e.,  $B=\mathbb{X}$. It turns out that the case $B\subsetneq \mathbb{X}$ is more involved because we do not have results that characterize $\mathcal{A}^s(L_p(B))$ by means of $\sigma_N$. In fact, $\sigma_N$ requires functions to be defined globally so that one would be forced to deal with boundary effects. On the other hand, $B$ itself may not be a diffusion measure space satisfying all required assumptions. We circumvent such difficulties by dealing with a modified  approximation space, for which we can construct a bit representation scheme.

\subsection{Characterization of local smoothness by local approximation rates}
Before we can discuss local smoothness, 
few technical details need to be introduced and we make use of $\mathcal{C}^\infty(\mathbb{X}):=\bigcap_{s>0}\mathcal{A}^s(L_\infty(\mathbb{X}))$:
\begin{definition}\label{def:smooth}
We say that $\mathbb{X}$ satisfies the \emph{smooth cut-off property} if for any $s>0$ and any two concentric balls $B',B$ with $B'\subsetneq B$ there is $\phi\in \mathcal{C}^\infty(\mathbb{X})$ such that $\phi$ equals $1$ on $B'$ and $\phi$ vanishes outside of $B$ with $|\phi(x)|\leq 1$, for all $x\in\mathbb{X}$.
\end{definition}
Note that any smooth manifold satisfies the smooth cut-off property. 
\begin{definition}
For $x\in\mathbb{X}$, the \emph{local approximation space in $x$} is denoted by $\mathcal{A}^s(X_p(\mathbb{X}),x)$ and defined as the collection of $f\in X_p(\mathbb{X})$ such that there is an open ball $B$ containing $x$ with $f\phi \in \mathcal{A}^s(X_p(\mathbb{X}))$, for all $\phi\in\mathcal{C}^\infty(\mathbb{X})$ with support in $B$. 
\end{definition}
It turns out that the approximation rate of $\sigma_N(f,\mu_N)$ characterizes the approximation class of $f$, at least when switching to dyadic numbers $N=2^n$, $n=1,2,\ldots$, cf.~\cite{Filbir:2010aa,Maggioni:2008fk,Mhaskar:2010kx,Mhaskar:2011uq}:
\begin{theorem}\label{th:alles}
Let $\mathbb{X}$ satisfy the smooth cut-off property. For $p=\infty$, suppose 
that $(\mu_{2^n})_{n=1}^\infty$ is a family of Marcinkiewicz-Zygmund quadrature measures of order $2^n$, respectively. For $1\leq p<\infty$, we choose $\mu_{2^n}=\mu$, $n=1,2,\ldots$. 
If $H$ is a low-pass filter, 
then the following points are equivalent:
\begin{itemize}
\item[(i)] $f\in \mathcal{A}^s(X_p(\mathbb{X}),x)$,


%
\item[(ii)] there is a ball $B$ centered at $x$ such that 
\begin{equation}\label{eq:sigma estimate characterization}
\|f-\sigma_{2^n}(f,\mu_{2^n})\|_{L_p(B)} \lesssim 2^{-ns}.
\end{equation}
\end{itemize}
\end{theorem}
Note that the generic constant in \eqref{eq:sigma estimate characterization} may depend on $x$ and $f$. Nonetheless, the above theorem characterizes local approximation by means of $\sigma_{2^n}$

The local approximation space $\mathcal{A}^s(X_p(\mathbb{X}),x)$, for $x\in\mathbb{X}$, is not endowed with any norm. 
In view of Theorem \ref{th:alles}, we fix some ball $B$ and introduce a new approximation space in the following. 
\begin{definition}\label{def:loc sob}
Let $B$ be a nonempty ball in $\mathbb{X}$. For $p=\infty$, suppose 
that $(\mu_{2^n})_{n=1}^\infty$ is a family of Marcinkiewicz-Zygmund quadrature measures of order $2^n$, respectively. For $1\leq p<\infty$, we choose $\mu_{2^n}=\mu$, $n=1,2,\ldots$. 
If $H$ is a low-pass filter, then we define the \emph{local approximation space in $B$} by
\begin{equation*}
\mathcal{A}^s(X_p(\mathbb{X}),B):=\{ f\in X_p(\mathbb{X}) :\|f\|_{\mathcal{A}^s(X_p(\mathbb{X}),B)}<\infty  \},
\end{equation*}
where
\begin{equation*}
\|f\|_{\mathcal{A}^s(X_p(\mathbb{X}),B)}:= \|f\|_{L_p(\mathbb{X})}+\sup_{n\geq 1}2^{ns}\| f - \sigma_{2^n}(f,\mu_{2^n})\|_{L_p(B)}.
\end{equation*}
\end{definition}
Note that if $p=\infty$, then the space $\mathcal{A}^s(X_p(\mathbb{X}),B)$ implicitly depends on the family $(\mu_{2^n})_{n=1}^\infty$ of Marcinkiewicz-Zygmund quadrature measures of order $2^n$, respectively. 
As opposed to $\mathcal{A}^s(X_p(B))$ defined in \eqref{eq:def of Sobolev}, the space $\mathcal{A}^s(X_p(\mathbb{X}),B)$ consists of functions defined globally that inherit approximation properties locally. By definition, we have 
\begin{equation}\label{eq:best approx for trivial case}
\|f-\sigma_{2^n}(f,\mu_{2^n})\|_{L_p(B)} \leq 2^{-ns}\|f\|_{\mathcal{A}^s(X_p(\mathbb{X}),B)}.
\end{equation}
Since $\sigma_{2^n}(f,\mu_{2^n})$ is a diffusion polynomial, we observe that 
\begin{equation*}
\mathcal{A}^s(X_p(\mathbb{X}),B)_{|B}\hookrightarrow \mathcal{A}^s(X_p(B)).
\end{equation*}
However, we cannot claim that the reverse embedding also holds. 

 It should be mentioned that $\sigma_{2^n}(f,\mu_{2^n})$ in \eqref{eq:best approx for trivial case} approximates $f$ locally but its definition needs global knowledge of $f$ or at least on $\supp(\mu_{2^n})$ if $p=\infty$. To enable the design of an approximation scheme that involves local information on $f$ exclusively, we define one more approximation space by using some cut-off function:
 \begin{definition}
 For some fixed $\phi\in \mathcal{C}^\infty(\mathbb{X})$, define
 \begin{equation}\label{eq:def of local sob new again}
 \mathcal{A}^s(X_p(\mathbb{X}),\phi) : = \{f\in X_p(\mathbb{X}) : f\phi\in \mathcal{A}^s(X_p(\mathbb{X}))\}
 \end{equation}
endowed with the norm $\|f\|_{\mathcal{A}^s(X_p(\mathbb{X}),\phi)}:=\|f\|_{L_p(\mathbb{X})} + \sup_{n\geq 1}2^{ns}E(f\phi,\Pi_{2^n},L_p(\mathbb{X}))$. 
\end{definition}
To study local approximation, choose two concentric balls $B',B$ with $B'\subsetneq B$. If $\mathbb{X}$ satisfies the smooth cut-off property, then we can fix some $\phi\in\mathcal{C}^\infty(\mathbb{X})$ that is one on $B'$ and zero outside of $B$. The Definition \eqref{eq:def of local sob new again} yields that $f\in  \mathcal{A}^s(X_p(\mathbb{X}),\phi)$ implies
\begin{equation*}
\|f-\sigma_{2^n}(f\phi,\mu_{2^n})\|_{L_p(B')}\leq \|f\phi - \sigma_{2^n}(f\phi,\mu_{2^n})\|_{L_p(\mathbb{X})}  \lesssim 2^{-ns}\|f\|_{\mathcal{A}^s(X_p(\mathbb{X}),\phi)}. 
\end{equation*}
In the subsequent section, we shall consider balls of radius $r$ for both spaces, 
\begin{align}
\overline{\mathcal{A}^s}(X_p(\mathbb{X}),B)&:=\{ f\in X_p(\mathbb{X}) :\|f\|_{\mathcal{A}^s(X_p(\mathbb{X}),B)}\leq 1\}, \label{eq:rad r sob} \\
\overline{\mathcal{A}^s}(X_p(\mathbb{X}),\phi)&:=\{ f\in X_p(\mathbb{X}) :\|f\|_{\mathcal{A}^s(X_p(\mathbb{X}),\phi)}\leq 1\},\label{eq:rad r sob 2}
\end{align}
and aim to develop approximation schemes requiring only few bits. 

 \subsection{Local bit-representation in approximation classes}
 
To design an approximation scheme for the spaces $\overline{\mathcal{A}^s}(X_p(\mathbb{X}),B)$ and $\overline{\mathcal{A}^s}(X_p(\mathbb{X}),\phi)$, let $B$ be a ball in $\mathbb{X}$ and let $B'\subset B$ be another nonempty ball concentric with $B$ and of radius strictly less. It will turn out that the following scheme enables us to approximate $f$ on $B'$. For some fixed $S>1$, we apply the quantization $I_{2^n}(f,\mu_{2^n},x)$ as in \eqref{eq:quantization2} but in place of \eqref{eq:discretized approx operator2} we define the local approximation by 
\begin{equation}\label{eq:discretized approx operator}
\sigma^\circ_{2^n}(f,\mu_{2^n},\nu_{2^n},B):= 2^{-nS} \int_{B} I_{2^n}(f,\mu_{2^n},x) \sum_{k=0}^\infty H(\frac{\lambda_k}{{2^{n+1}}})\varphi_k(x)d\nu_{2^n}(x)\varphi_k,
\end{equation}
We have the following result for the ball $\overline{\mathcal{A}^s}(X_p(\mathbb{X}),B)$ of the localized approximation space given by \eqref{eq:rad r sob}:
\begin{theorem}\label{th:alles2}
Suppose that $\mathbb{X}$ satisfies the smooth cut-off property and that $(\mu_{2^n})_{n=1}^\infty$ is a family of Marcinkiewicz-Zygmund quadrature measures of order ${2^n}$, respectively, if $p=\infty$. For $1\leq p<\infty$ we choose $\mu_{2^n}=\mu$, $n=1,2,\ldots$. Assume further that $H$ is a low-pass filter. Let $B, B'$ be two concentric balls, so that $B'\subsetneq  B$. We also suppose that $(\nu_{2^n})_{n=1}^\infty$ are Marcinkiewicz-Zygmund quadrature measures of order ${2^{n+1}}$ with $\#\supp(\nu_{2^n}) \lesssim 2^{n\alpha}$. For fixed $s>0$ and $S>\max(1,s)$, we apply the discretizations \eqref{eq:quantization2} and \eqref{eq:discretized approx operator}. Then there is a constant $c>0$ such that, for all $f\in \overline{\mathcal{A}^s}(X_p(\mathbb{X}),B)$,
\begin{equation}\label{eq:approx inequality in theorem}
\|f-\sigma^\circ_{2^n}(f,\mu_{2^n},\nu_{2^n},B)\|_{L_p(B')} \leq c  2^{-ns}
\end{equation}
holds. For $c2^{-ns} = \varepsilon\leq 1$ and $\varepsilon \leq 1$, the number of bits needed to represent all integers $\{I_{2^n}(f,\mu_{2^n},x): x\in  \supp(\nu_{2^n}) \cap B\}$ does not exceed a positive constant (independent of $\varepsilon$) times 
\begin{equation}\label{eq:almost optimal}
(1/\varepsilon)^{\alpha/s} (1+\log_2(1/\varepsilon) ).
\end{equation}
\end{theorem}
\begin{proof}[Proof of Theorem \ref{th:alles2}]
For $f\in \overline{\mathcal{A}^s}(X_p(\mathbb{X}),B)$, we use the localization property \eqref{eq:localization result!} with $r=\alpha+S$ and the embedding $L_p\hookrightarrow L_1$ in the compact case to derive, for $y\in B'$, 
\begin{align}
 \int_{\mathbb{X}\setminus B} \big|\sigma_{2^n}(f,\mu_{2^n},x)K_{2^{n+1}}(x,y) \big| d|\nu_{2^n}|(x)
 &\lesssim  2^{-nS}\| \sigma_{2^n}(f,\mu_{2^n}) \|_{|\nu_{2^n}|,L_p}\nonumber\\
 & \lesssim 2^{-nS}\| \sigma_{2^n}(f,\mu_{2^n}) \|_{|\mu_{2^n}|,L_p}.\label{eq:noch eine halt}
\end{align}
The latter estimate holds because both $(\nu_{2^n})_{n=1}^\infty$ and $(\mu_{2^n})_{n=1}^\infty$ are families of Marcinkiewicz-Zygmund measures. 
 The quantization \eqref{eq:quantization2} immediately yields 
\begin{equation}\label{eq:quantization immediately yields}
|\sigma_{2^n}(f,\mu_{2^n},x)-2^{-nS}I_{2^n}(f,\mu_{2^n},x)|\leq 2^{-nS},\quad\text{for all $x\in\supp(\nu_{2^n})$.}
\end{equation}
By using \eqref{eq:quantization immediately yields}, \eqref{eq:for later use}, and \eqref{eq:one estimate}, we derive
\begin{align*}
\|\sigma_{2^n}(f,\mu_{2^n})-\sigma^\circ_{2^n}(f,\mu_{2^n},\nu_{2^n},B)\|_{L_p(B')} & \\
&\hspace{-3cm}= \big\|\sigma_{2^n}(f,\mu_{2^n})-\int_B  2^{-nS} I_{2^n}(f,\mu_{2^n},x)K_{2^{n+1}}(x,\cdot)d\nu_{2^n}(x) \big\|_{L_p(B')} \\
&\hspace{-3cm} \lesssim \big\|\sigma_{2^n}(f,\mu_{2^n})-\int_B  \sigma_{2^n}(f,\mu_{2^n},x) K_{2^{n+1}}(x,\cdot)d\nu_{2^n}(x) \big\|_{L_p(B')} + 2^{-nS}.
\end{align*}
Next, we make use of \eqref{eq:for later use} and \eqref{eq:noch eine halt} with \eqref{eq:sigma Lp} to obtain
\begin{align*}
\|\sigma_{2^n}(f,\mu_{2^n})-\sigma^\circ_{2^n}(f,\mu_{2^n},B)\|_{L_p(B')}& \lesssim \big\|\int_{\mathbb{X}\setminus B} \sigma_{2^n}(f,\mu_{2^n},x) K_{2^{n+1}}(x,\cdot) d\nu_{2^n}(x)\big\|_{L_p(B')} +
2^{-nS} \\ 
& \lesssim 2^{-nS}\|f\|_{|\mu_{2^n}|,L_p}+2^{-nS}\lesssim 2^{-nS}.
\end{align*}
Here, it is important that we assume $\mu_{2^n}=\mu$, for $1\leq p<\infty$, so that $\|f\|_{|\mu_{2^n}|,L_p}\lesssim  \|f\|_{L_p(\mathbb{X})}$ holds for the entire range $1\leq p\leq \infty$. Note that $\|f\|_{L_p(\mathbb{X})}\leq 1$ follows from $f\in \overline{\mathcal{A}^s}(X_p(\mathbb{X}),B)$. The triangle inequality with \eqref{eq:best approx for trivial case} and the above estimate yield 
\begin{align*}
\|f-\sigma^\circ_{2^n}(f,\mu_{2^n},B)\|_{L_p(B')} &\lesssim  \|f-\sigma_{2^n}(f,\mu_{2^n})\|_{L_p(B')} \\ 
& \qquad \qquad + \|\sigma_{2^n}(f,\mu_{2^n})-\sigma^\circ_{2^n}(f,\mu_{2^n},\nu_{2^n},B) \|_{L_p(B')} \\
& \lesssim 2^{-ns}+2^{-nS}\lesssim  2^{-ns},
\end{align*}
which verifies \eqref{eq:approx inequality in theorem}. 

For the remaining part, we can follow the lines of the proof of Theorem \ref{th:alles22}. 
\end{proof}

Note that Theorem  \ref{th:alles2} still  requires global knowledge of $f$ because we need to build $\sigma_{2^n}(f,\mu_{2^n})$. For the sake of completeness, we use a cut-off function to feed in local information only:
\begin{theorem}\label{theorem last one}
Under the same assumption as in Theorem \ref{th:alles22}, let $\phi\in \mathcal{C}^\infty(\mathbb{X})$ be one on $B'$ and zero outside of $B$. Then there is a constant $c>0$ such that, for $f\in \overline{\mathcal{A}^s}(X_p(\mathbb{X}),\phi)$,
\begin{equation}\label{eq:approx inequality in theorem 2}
\begin{split}
\|f-\sigma^\circ_{2^n}(f\phi,\mu_{2^n},\nu_{2^n})\|_{L_p(B')} &\leq c 2^{-ns}
\end{split}
\end{equation}
holds. For $c2^{-ns} = \varepsilon \leq 1$, the number of bits needed to represent all integers $\{I_{2^n}(f\phi,\mu_{2^n},x): x\in \supp(\nu_{2^n}) \}$ does not exceed a positive constant  (independent of $\varepsilon$) times 
\begin{equation}\label{eq:almost optimal 2}
(1/\varepsilon)^{\alpha/s}(1+ \log_2(1/\varepsilon) ).
\end{equation}
\end{theorem}
\begin{proof}
The smooth cut-off property of $\phi$ yields that $f\in \overline{\mathcal{A}^s}(X_p(\mathbb{X}),\phi)$ leads to $f\phi\in \overline{\mathcal{A}^s}(X_p(\mathbb{X}))$, due to $|\phi(x)|\leq 1$, for all $x\in\mathbb{X}$.  Thus, Theorem \ref{th:alles22} applied to $f\phi$ implies \eqref{eq:approx inequality in theorem 2} and \eqref{eq:almost optimal 2}. 
\end{proof}

\appendix 

\section{Summary of the technical assumptions}\label{sec assumptions}
The asymptotic bounds on the metric entropy in Theorem \ref{th:op bit} hold for any diffusion measure space $\mathbb{X}$ as introduced in Definition \ref{dimmdef} provided that the restrictions $\{\varphi_{k|B}\}_{k=0}^\infty$ are linearly independent on the ball $B\subset\mathbb{X}$. Our computational scheme that matches this bound up to a logarithmic factor needs few additional technical assumptions that are distributed within the present paper. Here, we list all the required assumptions for the sake of completeness:
\begin{enumerate}[(I)]
\item \label{item 1}$\mathbb{X}$ is a diffusion measure space (Definition \ref{dimmdef}),
\item\label{item2.5} $\epsilon_N$ defined in \eqref{eq:def of eps} satisfies $N^m\epsilon_N \rightarrow 0$ as $N\rightarrow \infty$, for all $m>0$, 
\item\label{item 3} there is a family $(\nu_N)_{N=1}^\infty$ of Marcinkiewicz-Zygmund quadrature measures  of order $2N$, respectively, satisfying $\#\supp(\nu_N) \lesssim N^\alpha$ (Definitions \ref{def:quad}, \ref{def:MZ} and \eqref{eq:for later use}),
\item\label{item 4} the smooth cut-off property holds (Definition \ref{def:smooth}).
\end{enumerate}
Condition \eqref{item 1} is the general framework, and the additional conditions are further technical details. Note that \eqref{item 4} is only needed in Section \ref{sec:local}, and it is well-known that this holds for smooth manifolds. All of the above conditions hold for compact homogeneous manifolds, e.g., the sphere and the Grassmann manifold, if the function system $\{\varphi_k\}_{k=0}^\infty$ are eigenfunctions of the Laplace operator, cf.~\cite{Geller:2011fk}. Moreover, it was pointed out in \cite{Chui:2013fk} that the conditions are also satisfied for smooth compact Riemannian manifolds with nonnegative Ricci curvature, see also Remark \ref{rem 2.2}. Families of Marcinkiewicz-Zygmund quadrature measures $(\mu_N)_{N=1}^\infty$ were then constructed in \cite{Filbir:2011fk}, such that  $\#\supp(\mu_N) \asymp N^\alpha$. 

\section*{Acknowledgment}
M.E.~has been funded by the Vienna Science and Technology Fund (WWTF) through project VRG12-009. 
The authors thank H.~N.~Mhaskar for many fruitful discussions.

\bibliography{../biblio_ehler2}

\end{document}